\documentclass[11pt, reqno]{amsart}
\usepackage{amssymb,latexsym,amsmath,amsfonts,mathdots,enumitem}
\usepackage{latexsym}
\usepackage[mathscr]{eucal}
\usepackage{colortbl,xcolor}
\usepackage{lmodern}
\usepackage{sansmathaccent}
\usepackage[latin1]{inputenc}
\usepackage{tikz}
\usepackage{physics}
\usetikzlibrary{shapes,arrows}
\allowdisplaybreaks
\usetikzlibrary{matrix,calc,shapes,arrows,positioning}
\pdfmapfile{+sansmathaccent.map}

\numberwithin{equation}{section}
\theoremstyle{plain}

\newtheorem{thm}{Theorem}[section]

\newtheorem{lem}[thm]{Lemma}
\newtheorem{prop}[thm]{Proposition}
\theoremstyle{definition}
\newtheorem{defn}[thm]{Definition}
\newtheorem{rem}[thm]{Remark}

\newtheorem{prob}[thm]{Problem}
\numberwithin{equation}{section}

\def\beq{\begin{eqnarray}}
	\def\eeq{\end{eqnarray}}
\def\beqa{\begin{eqnarray*}}
	\def\eeqa{\end{eqnarray*}}

\def\Ran{\operatorname{Ran}}

\def\beqn{\begin{equation}}
	\def\eeqn{\end{equation}}

\def\mg#1{}

\def\Ran{\operatorname{Ran}}

\renewcommand{\epsilon}{\varepsilon}
\renewcommand{\phi}{\varphi}

\renewcommand{\bf}[1]{\textbf{#1}}

\renewcommand{\sf}[1]{\textsf{#1}}

\numberwithin{equation}{section}
\allowdisplaybreaks[4] 

\setlist[enumerate]{font=\upshape,noitemsep, topsep=0pt} 
\setlist[itemize]{noitemsep, topsep=0pt}

\begin{document}
	
	\title{The Families of $\Gamma_n$-Contractions Having Non-commutative Fundamental Operators}
	\author{Bhaskar Paul}
	
	\address[B. Paul]{Department of Mathematics, IIT Bhilai, 6th Lane Road, Jevra, Chhattisgarh 491002}
	\email{B. Paul:bhaskarpaul@iitbhilai.ac.in }
	
	\subjclass[2010]{47A13, 47A20}
	
	\keywords{Symmetrized polydisc, unitary dilation, fundamental operators, complete spectral set}
	
	\maketitle
	
	\begin{abstract}
		We call an $n$-tuple of $2 \times 2$ matrices a ``\textit{$2 \times 2$ matrix $\Gamma_n$-contraction}'' if its joint eigenvalues (joint spectrum for a commuting tuple of matrices) are contained in $\Gamma_n$. In this short note on $\Gamma_n$-contraction, we produce two families of $2 \times 2$ matrix $\Gamma_n$-contractions whose fundamental operators do not satisfy the commutativity conditions and
		\begin{equation*}
			\begin{aligned}
				F_iF^*_{n-j} - F_jF^*_{n-i} = F^*_{n-j}F_i - F^*_{n-i}F_j, \quad 1 \le i, j \le n-1,
			\end{aligned}
		\end{equation*}
		which shows that the commutativity as well as the conditions on fundamental operators mentioned above are sufficient, but not necessary, for the existence of $\Gamma_n$-isometric dilation.
	\end{abstract}

	\section{Introduction}\label{Intro}
	
	Suppose $\Omega$ is a compact subset in $\mathbb{C}^n$ and $\mathcal{O}(\Omega)$ is the algebra of bounded holomorphic functions in some neighbourhood of $\Omega$. The algebra $\mathcal{O}(\Omega)$ is equipped with the sup norm $||\, . \,||_{\infty, \Omega}$, defined by
	\begin{equation*}
		\begin{aligned}
			||f||_{\infty, \Omega} &= \sup\{|f(z)| : z \in \Omega\}.
		\end{aligned}
	\end{equation*}
	A compact subset $\Omega$ of $\mathbb{C}^n$ is said to be a \textit{spectral set} for a commuting $n$-tuple of bounded operators $\mathbf{T} = (T_1, \dots, T_n)$ if $\sigma(\mathbf{T}) \subset \Omega$ and the homomorphism $\rho_{\mathbf{T}} : \mathcal{O}(\Omega) \to \mathcal{B}(\mathcal{H})$ defined by
	\begin{equation*}
		\begin{aligned}
			\rho_{\mathbf{T}}(1) &= I_{\mathcal{H}} \quad \text{and} \quad \rho_{\mathbf{T}}(z_i) = T_i, \quad 1 \le i \le n
		\end{aligned}
	\end{equation*}
	is contractive (see \cite{McCarthy 1, Nagy}). von Neumann introduced the notion of a spectral set for an operator on a Hilbert space. The following theorem says that the closed unit disc is a spectral set for every contraction.
	\begin{thm}[{\cite[Chapter 1, Corollary 1.2]{Nagy}}]\label{von Neumann Inequality}
		Let $T$ be a contraction on a
		Hilbert space. Then for any polynomial $p$,
		\begin{equation*}
			\begin{aligned}
				||p(T)|| &\le \sup\{|p(z)| : |z| \le 1\}.
			\end{aligned}
		\end{equation*}
	\end{thm}
	The following dilation theorem for a contraction is due to Sz.-Nagy.
	\begin{thm}[{\cite[Theorem 1.1]{paulsen}}]\label{Nagy Dilation}
		Let $T$ be a contraction operator on Hilbert space $\mathcal{H}$.  Then there is a Hilbert space $\mathcal{K}$ containing $\mathcal{H}$ as a subspace and a unitary operator $U$ on $\mathcal{K}$ such that
		\begin{equation}\label{Power Dilation}
			\begin{aligned}
				T^n &= P_{\mathcal{H}}U^n|_{\mathcal{H}}, \quad n \in \mathbb{N} \cup \{0\}.
			\end{aligned}
		\end{equation}
	\end{thm}
	The dilation presented in Theorem \ref{Nagy Dilation} is called \textit{power dilation} of $T$. An explicit construction of a dilation of this type of unitary for a given contraction is given by Sch\"{a}ffer (see \cite{Schaffer}).
	
	Let $F = ((f_{ij}))$ be a matrix-valued polynomial defined on $\Omega$. We say that $\Omega$ is a \emph{complete spectral set} for $\mathbf T$ if
	\begin{equation*}
		\begin{aligned}
			||F(\mathbf{T})|| &\le ||F||_{\infty, \Omega}
		\end{aligned}
	\end{equation*}
	for every $F\in \mathcal{O}(\Omega) \otimes M_{k \times k}(\mathbb{C}), k \ge 1$. A commuting $n$-tuple $\mathbf T$ for which $\Omega$ is a spectral set is said to admit a $\partial \Omega$-normal dilation if there exist a Hilbert space $\mathcal{K} \supseteq \mathcal{H}$ and a commuting $n$-tuple of normal operators
	$\mathbf{N} = (N_1, \dots, N_n)$
	on $\mathcal{K}$ satisfying $\sigma(\mathbf{N}) \subseteq \partial \Omega$ such that
	\begin{equation*}
		\begin{aligned}
			P_{\mathcal H}f(\mathbf{N})|_{\mathcal H} &= f(\mathbf{T}), \quad f \in \mathcal{O}(\Omega).
		\end{aligned}
	\end{equation*}
	The following theorem due to Arveson (see \cite[Corollary, p. 279]{Arveson}) connects the notion of complete spectral set and $\partial \Omega$-normal dilation.
	\begin{thm}\label{Thm Arveson}
		Let $\mathbf{T} = (T_1, \dots, T_n)$ have K as a spectral set. Then $\mathbf{T}$ has $K$ as a complete spectral set if and only if $\mathbf{T}$ has a $\partial K$-normal dilation.
	\end{thm}
	
	Agler in \cite[Theorem 1.9]{Agler} proved the following result related to the complete contractivity of a unital algebra homomorphism from $\mathcal{O}(K)$ to $M_2(\mathbb{C})$, which we recall below. It plays a crucial role in proving our main results.
	
	\begin{thm}\label{Thm Agler}
		Let $K$ be a compact subset of $\mathbb{C}^n$. A unital $(\Pi(1) = 1)$ contractive $(||\Pi(f)|| \le ||f||)$ algebra homomorphism $\Pi : \mathcal{O}(K) \to M_2(\mathbb{C})$ is completely contractive.
	\end{thm}
	
	The domain $\mathbb{G}_n (n \ge 2)$ is the image of $\mathbb{D}^n$ under the \textit{symmetrization map} $\mathbf{s} := (s_1, \dots, s_n) : \mathbb{C}^n \to \mathbb{C}^n$ defined by
	\begin{equation*}
		\begin{aligned}
			s_i(z_1, \dots, z_n) &=
			\sum_{1 \le k_1 < \dots < k_i \le n}
			z_{k_1} \cdots z_{k_i},
			\quad 1 \le i \le n.
		\end{aligned}
	\end{equation*}
	The closure of $\mathbb{G}_n$ is denoted and defined by $\Gamma_n = \mathbf{s}(\overline{\mathbb{D}}^n)$. The detailed study on $\mathbb{G}_n$ and $\Gamma_n$ can be found in \cite{Costara}. Biswas and Shyam Roy have done operator theory on $\Gamma_n$ in \cite{Biswas}. From there we recall the definitions of $\Gamma_n$-contraction, $\Gamma_n$-unitary, and $\Gamma_n$-isometry.
	\begin{defn}
		Let $(S_1, \dots, S_n)$ be a commuting $n$-tuple of bounded operators acting on some Hilbert space $\mathcal{H}$. We call $(S_1, \dots, S_n)$ is
		\begin{enumerate}
			\item a $\Gamma_n$-contraction if $\Gamma_n$ is a spectral set for $(S_1, \dots, S_n)$.
			
			\item a $\Gamma_n$-unitary if $S_1, \dots, S_n$ are normal operators on $\mathcal{H}$ and the joint spectrum $\sigma(S_1, \dots, S_n) \subseteq b\Gamma_n$.
			
			\item a $\Gamma_n$-isometry if there exists a Hilbert space $\mathcal{K} \supseteq \mathcal{H}$ and a $\Gamma_n$-unitary $(N_1, \dots, N_n)$ acting on $\mathcal{K}$ such that $\mathcal{H}$ is a common invariant subspace for $N_1, \dots, N_n$ and $S_i = N_i|_{\mathcal{H}}$ for $1 \le i \le n$. 
		\end{enumerate}
	\end{defn}
	
	The following theorem is a straightforward consequence of Theorem \ref{Thm Agler}. Since the existence of a $\Gamma_n$-isometric dilation ensures the existence of a $\Gamma_n$-unitary dilation, we can state the following theorem for $\Gamma_n$-isometric dilation.
	\begin{thm}\label{2 times 2 Gamma_n-Contraction}
		Every $2 \times 2$ matrix $\Gamma_n$-contraction has a $\Gamma_n$-isometric dilation.
	\end{thm}
	
	Let $T$ be a contraction on Hilbert space $\mathcal{H}$. Define the \textit{defect operator} and \textit{defect space} of $T$ by
	\begin{equation*}
		\begin{aligned}
			D_T &= (I - T^*T)^{1/2} \quad \text{and} \quad \mathcal{D}_T = \overline{\Ran}D_T
		\end{aligned}
	\end{equation*}
	respectively. For $1 \le i \le n - 1$, the \textit{fundamental equations} for a $\Gamma_n$-contraction $(S_1, \dots, S_n)$ are the following:
	\begin{equation}\label{Gamma_n Fundamental}
		\begin{aligned}
			S_i - S^*_{n-i}S_n = D_{S_n}F_iD_{S_n} \quad \text{and} \quad S_{n-i} - S^*_iS_n = D_{S_n}F_{n-i}D_{S_n}.
		\end{aligned}
	\end{equation}
	We refer to $F_1, \dots, F_{n-1}$ the \textit{fundamental operators} for $(S_1, \dots, S_n)$. The detailed study related to the fundamental equations and operators can be found in \cite{S. Pal, A. Pal}. We recall \cite[Theorem 4.4]{A. Pal}, which establishes the existence and uniqueness of the solutions of the fundamental equations.
	
	\begin{thm}[Existence and Uniqueness]\label{Existence and Uniqueness}
		For $n \ge 2$, let $(S_1, \dots, S_n)$ be a $\Gamma_n$-contraction on a Hilbert space $\mathcal{H}$. Then there are unique operators $F_1, \dots, F_{n-1} \in \mathcal{B}(D_{S_n})$ such that $S_i - S^*_{n-i}S_n = D_{S_n}F_iD_{S_n}$ and $S_{n-i} - S^*_iS_n = D_{S_n}F_{n-i}D_{S_n}, F_i, F_{n-i} \in \mathcal{B}(D_{S_n})$, for $1 \le i \le n - 1$. Moreover, $w(F_i + F_{n-i}z) \le \binom{n-1}{i} + \binom{n-1}{n-i}$ for all $z \in \mathbb{T}$.
	\end{thm}
	
	Let $F_1, \dots, F_{n-1} \in \mathcal{B}(D_{S_n})$ be the fundamental operators of a $\Gamma_n$-contraction that satisfy the following condition:
	\begin{equation}\label{Commutativity Condition}
		\begin{aligned}
			F_iF_j = F_jF_i \quad \text{and} \quad F_iF^*_{n-j} - F_jF^*_{n-i} = F^*_{n-j}F_i - F^*_{n-i}F_j, \quad 1 \le i, j \le n-1.
		\end{aligned}
	\end{equation}
	These conditions play a vital role in constructing the conditional dilation of a $\Gamma_n$-contraction (see \cite[Theorem 6.6, Corollary 6.7]{A. Pal}). From here one can think of the following problem.
	\begin{prob}\label{Prob 1}
		\textit{Can the conditions in \eqref{Commutativity Condition} be relaxed or removed while obtaining a $\Gamma_n$-isometric dilation? Equivalently, whether the conditions in \eqref{Commutativity Condition} are necessary for the existence of a $\Gamma_n$-isometric dilation.}
	\end{prob}
	
	Mandal and Pal in \cite[Example 2.12, Example 3.3]{Mandal} showed that the condition
	\begin{equation*}
		\begin{aligned}
			F_iF^*_{n-j} - F_jF^*_{n-i} = F^*_{n-j}F_i - F^*_{n-i}F_j, \quad 1 \le i, j \le n-1
		\end{aligned}
	\end{equation*}
	is not necessary, in general. But it remains an open question whether the commutativity of the fundamental operators is necessary. Bhowmik in \cite[Theorem 2.3]{Bhowmik} demonstrated a family of tetrablock contractions (see \cite{Bhattacharyya}) that possess tetrablock-unitary dilation, while their fundamental operators are non-commutative. Motivated by his work, we present a family of $2 \times 2$ matrix $\Gamma_n$-contractions that extend to $\Gamma_n$-isometric dilations, which gives an answer to Problem \ref{Prob 1}.

	\section{The Main Results}
	
	In this section, we produce two families of $2 \times 2$ matrix $\Gamma_n$-contractions that have non-commuting fundamental operators. Before going into the main result, we recall a result given by Knese.
	\begin{prop}[{\cite[Proposition 1.1]{Knese}}]\label{Prop Knese}
		If $f : \mathbb{D}^n \to \mathbb{D}$ is holomorphic, then
		\begin{equation*}
			\begin{aligned}
				\sum_{i=1}^{n} (1 - |z_i|^2)\left|\frac{\partial f}{\partial z_i}(z)\right| &\le 1 - |f(z)|^2
			\end{aligned}
		\end{equation*}
		for any $z = (z_1, \dots, z_n) \in \mathbb{D}^n$.
	\end{prop}
	
	We denote the set $\{(rz_1, \dots, rz_n) : (z_1, \dots, z_n) \in \mathbb{G}_n\}$ by $r\mathbb{D}^n$. Notice that $\mathbf{0} = (0, \dots, 0) \in \mathbb{G}_n$ as it satisfies the $|f(z)| < 1$. Then $\mathbf{0}$ is an interior point of $\mathbb{G}_n$. Thus, there exists $r \in (0, 1)$ such that $r\overline{\mathbb{D}}^n \subseteq \mathbb{G}_n$. Hence, we obtain the following lemma.
	\begin{lem}\label{Lem 1}
		There exists an $r \in (0, 1)$ such that $r\overline{\mathbb{D}}^n \subseteq \mathbb{G}_n$.
	\end{lem}
	
	Let us consider the following $n$-tuples of matrices.
	\begin{enumerate}
		\item For $n$ odd, $(S_1, \dots, S_n)$, where $S_i$ for $1 \le i \le n-1$ and $S_n$ are as follows:
		\begin{equation*}
			\begin{aligned}
				S_i &=
				\begin{cases}
					\left[
					\begin{matrix}
						0 & \delta_i\\
						0 & 0
					\end{matrix}\right], & 1 \le i \le \frac{n-1}{2},\\
					\left[
					\begin{matrix}
						0 & \delta_{n-i}\\
						0 & 0
					\end{matrix}\right], & \frac{n+1}{2} \le i \le n-1,
				\end{cases}
				\quad \text{and} \quad
				S_n =
				\left[
				\begin{matrix}
					0 & \delta\\
					0 & 0
				\end{matrix}\right].
			\end{aligned}
		\end{equation*}
		
		\item For $n$ even, $(S_1, \dots, S_n)$, where $S_i$ for $1 \le i \le n-1$ and $S_n$ are as follows:
		\begin{equation*}
			\begin{aligned}
				S_i &=
				\begin{cases}
					\left[
					\begin{matrix}
						0 & \delta_i\\
						0 & 0
					\end{matrix}\right], & 1 \le i \le \frac{n-2}{2},\\
					\left[
					\begin{matrix}
					0 & \delta_{n/2}\\
					0 & 0
					\end{matrix}\right], & i = n/2,\\
					\left[
					\begin{matrix}
						0 & \delta_{n-i}\\
						0 & 0
					\end{matrix}\right], & \frac{n+2}{2} \le i \le n-1,
				\end{cases}
				\quad \text{and} \quad
				S_n =
				\left[
				\begin{matrix}
					0 & \delta\\
					0 & 0
				\end{matrix}\right].
			\end{aligned}
		\end{equation*}
	\end{enumerate}
	
	We now prove one of our main result. In this theorem we provide a class of $\Gamma_n$-contraction whose fundamental operators do not commute.
	\begin{thm}\label{Main Theorem 1}
		Let $r$ be as in Lemma \ref{Lem 1}. Then the following hold:
		\begin{enumerate}
			\item For $n$ odd, the $n$-tuple of matrices $(S_1, \dots, S_n)$ is a $\Gamma_n$-contraction such that the fundamental operators do not satisfy all the conditions in \eqref{Commutativity Condition}, whenever $0 < |\delta_i|, |\delta| < r$ and $|\delta_i| \ne |\delta_j|$ for $1 \le i, j \le \frac{n-1}{2}$ or $\frac{n+1}{2} \le i, j \le n-1$ or $j \ne n-i$.
			
			\item For $n$ even, the $n$-tuple of matrices $(S_1, \dots, S_n)$ is a $\Gamma_n$-contraction such that the fundamental operators do not satisfy all the conditions in \eqref{Commutativity Condition}, whenever $0 < |\delta_i|, |\delta| < r$ and $|\delta_i| \ne |\delta_j|$ for $1 \le i, j \le \frac{n-2}{2}$ or $\frac{n+2}{2} \le i, j \le n-1$ or $j \ne n-i$ and $|\delta_{n/2}| \ne |\delta_i|$ for $i \ne n/2$.
		\end{enumerate}
	\end{thm}
	
	\begin{proof}
		We only prove $(1)$. The proof of $(2)$ follows in the same line to that of $(1)$.
		
		Observe that $S_1, \dots, S_n$ are commuting matrices in $M_2(\mathbb{C})$, as the product of any two is zero. We know that the joint spectrum of a tuple of commuting matrices is the set of joint eigenvalues of the matrices. Since $S_1, \dots, S_n$ are nilpotent, the only eigenvalue of them is $0$. Thus, the only joint eigenvalue of the tuple $(S_1, \dots, S_n)$ is $(0, \dots, 0)$ with a joint eigenvector
		$\left[\begin{smallmatrix}
			x\\
			y
		\end{smallmatrix}\right]$ in $\mathbb{C}^2$. By routine computation, one can see that
		$\left[\begin{smallmatrix}
			1\\
			0
		\end{smallmatrix}\right]$ is one of the joint eigenvectors of $(S_1, \dots, S_n)$. Hence, $\sigma(S_1, \dots, S_n) = \{(0, \dots, 0)\} \subseteq \Gamma_n$.
		
		Let $f \in \mathcal{O}(\Gamma_n)$ such that $f(0, \dots, 0) = \alpha \in \mathbb{D} \setminus \{0\}$ and $||f||_{\infty, \Gamma_n} = 1$. Consider the Blaschke factor $\Psi_{\alpha}(z) = e^{i\theta}\frac{z - \alpha}{1 - \overline{\alpha}z}, z \in \mathbb{D}$ for some $\theta \in \mathbb{R}$. Since $\Psi_{\alpha} \in \operatorname{Aut}(\mathbb{D})$, we see that $\Psi_{\alpha} \circ f$ is an arbitrary holomorphic function from $\mathbb{G}_n$ to $\overline{\mathbb{D}}$. Then $(\Psi_{\alpha} \circ f)(0, \dots, 0) = 0$. We therefore assume that $f(0, \dots, 0) = 0$.
		
		As $S^2_i = S^2_n = 0$, by Taylor series expansion of $f$ and the holomorphic functional calculus, we obtain
		\begin{equation*}
			\begin{aligned}
				f(S_1, \dots, S_n) &=
				\left[
				\begin{matrix}
					0 & \Lambda\\
					0 & 0
				\end{matrix}\right],
			\end{aligned}
		\end{equation*}
		where $\Lambda = \displaystyle\sum_{i=1}^{\frac{n-1}{2}} \delta_i \frac{\partial f}{\partial z_i}(0, \dots, 0) + \displaystyle\sum_{i=\frac{n+1}{2}}^{n-1} \delta_{n-i} \frac{\partial f}{\partial z_i}(0, \dots, 0) + \delta \frac{\partial f}{\partial z_n}(0, \dots, 0)$. Since $||f||_{\infty, \Gamma_n} = 1$ and $0 < |\delta_i|, |\delta| < r$, we have
		\begin{equation}\label{Main 1}
			\begin{aligned}
				||f(S_1, \dots, S_n)||
				&= \left|\displaystyle\sum_{i=1}^{\frac{n-1}{2}} \delta_i \frac{\partial f}{\partial z_i}(0, \dots, 0) + \displaystyle\sum_{i=\frac{n+1}{2}}^{n-1} \delta_{n-i} \frac{\partial f}{\partial z_i}(0, \dots, 0) + \delta \frac{\partial f}{\partial z_n}(0, \dots, 0)\right|\\
				&\le r\left(\displaystyle\sum_{i=1}^{\frac{n-1}{2}} \left|\frac{\partial f}{\partial z_i}(0, \dots, 0)\right| + \displaystyle\sum_{i=\frac{n+1}{2}}^{n-1} \left|\frac{\partial f}{\partial z_i}(0, \dots, 0)\right| + \left|\frac{\partial f}{\partial z_n}(0, \dots, 0)\right|\right)\\
				&\le 1 - |f(0, \dots, 0)|^2 = 1 = ||f||_{\infty, \Gamma_n}.
			\end{aligned}
		\end{equation}
		The second inequality in \eqref{Main 1} follows by applying Proposition \ref{Prop Knese} on $f(rz_1, \dots, rz_n)$.
		
		Whenever $||f||_{\infty, \Gamma_n} \neq 1$, then by applying the same technique to the function $\widetilde{f} = \frac{f}{||f||_{\infty, \Gamma_n}}$ we get,
		\begin{equation}\label{Main 2}
			\begin{aligned}
				||\widetilde{f}(S_1, \dots, S_n)||
				&\le 1.
			\end{aligned}
		\end{equation}
		As $||f||_{\infty, \Gamma_n}$ is a scalar constant, $\frac{\partial f}{\partial z_i}(0, \dots, 0)$ is independent of $||f||_{\infty, \Gamma_n}$. Therefore, from \eqref{Main 2}, we deduce that
		\begin{equation*}
			\begin{aligned}
				||f(S_1, \dots, S_n)|| &\le ||f||_{\infty, \Gamma_n}.
			\end{aligned}
		\end{equation*}
		This shows that $(S_1, \dots, S_n)$ is a $\Gamma_n$-contraction such given $\delta_i$'s and $\delta$.
		
		First we compute the defect operator of $S_n$. Note that
		\begin{equation*}
			\begin{aligned}
				D^2_{S_n} &= I - S^*_nS_n =
				\left[
				\begin{matrix}
					1 & 0\\
					0 & 1
				\end{matrix}\right] -
				\left[
				\begin{matrix}
					0 & 0\\
					\overline{\delta} & 0
				\end{matrix}\right]
				\left[
				\begin{matrix}
					0 & \delta\\
					0 & 0
				\end{matrix}\right] =
				\left[
				\begin{matrix}
					1 & 0\\
					0 & 1 - |\delta|^2
				\end{matrix}\right].
			\end{aligned}
		\end{equation*}
		Thus, we have
		\begin{equation*}
			\begin{aligned}
				D_{S_n} &=
				\left[
				\begin{matrix}
					1 & 0\\
					0 & \sqrt{1 - |\delta|^2}
				\end{matrix}\right].
			\end{aligned}
		\end{equation*}
		Let $F_1, \dots, F_{n-1}$ be the fundamental operators of $(S_1, \dots, S_n)$. Since $D_{S_n}$ is invertible and $S_i = S_{n-i}$, by \eqref{Gamma_n Fundamental} we obtain
		\begin{equation*}
			\begin{aligned}
				F_i &= D^{-1}_{S_n}(S_i - S^*_{n-i}S_n)D^{-1}_{S_n} = D^{-1}_{S_n}(S_{n-i} - S^*_iS_n)D^{-1}_{S_n} = F_{n-i}, \quad 1 \le i \le \frac{n-1}{2}.
			\end{aligned}
		\end{equation*}
		By routine computation we deduce that
		\begin{equation*}
			\begin{aligned}
				F_i &= 
				\left[
				\begin{matrix}
					0 & \frac{\delta_i}{\sqrt{1 - |\delta|^2}}\\
					0 & -\frac{\overline{\delta}_i\delta}{1 - |\delta|^2}
				\end{matrix}\right] = F_{n-i}, \quad 1 \le i \le \frac{n-1}{2}.
			\end{aligned}
		\end{equation*}
		
		Thus, it follows that $F_i$ commutes with $F_{n-i}$, whereas for $i, j$ satisfying at least one of $1 \le i, j \le \frac{n-1}{2}, \,\, \frac{n+1}{2} \le i, j \le n-1$, and $j \ne n-i$,
		\begin{equation*}
			\begin{aligned}
				F_iF_j &=
				\left[
				\begin{matrix}
					0 & -\frac{\delta_i\overline{\delta}_j\delta}{(1 - |\delta|^2)^{3/2}}\\
					0 & -\frac{\overline{\delta}_i\overline{\delta}_j\delta^2}{(1 - |\delta|^2)^2}
				\end{matrix}\right] \ne
				\left[
				\begin{matrix}
					0 & -\frac{\delta_j\overline{\delta}_i\delta}{(1 - |\delta|^2)^{3/2}}\\
					0 & -\frac{\overline{\delta}_j\overline{\delta}_i\delta^2}{(1 - |\delta|^2)^2}
				\end{matrix}\right] = F_jF_i,
			\end{aligned}
		\end{equation*}
		showing that not all $F_i$'s commute with each other. Also one can check that
		\begin{equation*}
			\begin{aligned}
				F_iF^*_{n-j} - F_jF^*_{n-i} &=
				\left[
				\begin{matrix}
					\frac{\delta_i\overline{\delta}_{n-j} - \delta_j\overline{\delta}_{n-i}}{1 - |\delta|^2} & \frac{(\delta_j\delta_{n-i} - \delta_i\delta_{n-j})\overline{\delta}}{(1 - |\delta|^2)^{3/2}}\\
					\frac{(\overline{\delta}_j\overline{\delta}_{n-i} - \overline{\delta}_i\overline{\delta}_{n-j})\delta}{(1 - |\delta|^2)^{3/2}} & \frac{\overline{\delta}_i\delta_{n-j} - \overline{\delta}_j\delta_{n-i}}{(1 - |\delta|^2)^2}
				\end{matrix}\right]\\
				&\ne
				\left[
				\begin{matrix}
					0 & 0\\
					0 & \frac{\overline{\delta}_{n-j}\delta_i - \overline{\delta}_{n-i}\delta_j}{1 - |\delta|^2}
					+ \frac{(\delta_{n-j}\overline{\delta}_i - \delta_{n-i}\overline{\delta}_j)|\delta|^2}{(1 - |\delta|^2)^2}
				\end{matrix}\right]\\
				&= F^*_{n-j}F_i - F^*_{n-i}F_j, \quad j \ne n-i,
			\end{aligned}
		\end{equation*}
		and for $j = n-i$,
		\begin{equation*}
			\begin{aligned}
				F_iF^*_{n-j} - F_jF^*_{n-i} &= F^*_{n-j}F_i - F^*_{n-i}F_j.
			\end{aligned}
		\end{equation*}
		Thus, it is clear from this particular family of $\Gamma_n$-contraction, that not all the conditions in \eqref{Commutativity Condition} are necessary to have a $\Gamma_n$-isometric dilation. But the conditions in \eqref{Commutativity Condition} are being satisfied when $j = n-i$. This completes the proof.
	\end{proof}
	
	We now provide a family of $2 \times 2$ matrix $\Gamma_n$-contraction which shows that none of the conditions in \eqref{Commutativity Condition} is necessary for a $\Gamma_n$-contraction to obtain an $\Gamma_n$-isometric dilation. Before proceeding to construct the family, we define the following nilpotent matrices:
	\begin{equation}\label{T_i}
		\begin{aligned}
			T_i &=
			\left[
			\begin{matrix}
				0 & \lambda_i\\
				0 & 0
			\end{matrix}\right], \quad \lambda_i = s_i(\delta_1, \dots, \delta_n), \quad |\delta_i| < r/M, \quad 1 \le i \le n,
		\end{aligned}
	\end{equation}
	where $M = \max\{k(i) : 1 \le i \le n\}, \,\, r$ is as in Lemma \ref{Lem 1} and $\delta_1, \dots, \delta_n$ are such complex numbers that $\lambda_i\overline{\lambda}_{n-j} \ne \lambda_j\overline{\lambda}_{n-i}$ for $1 \le i, j \le n-1$. One can easily notice that $(\lambda_1, \dots, \lambda_n) \in \Gamma_n$. This implies that
	\begin{equation}\label{Lambda_i, Lambda_n Estimate}
		\begin{aligned}
			|\lambda_i| < r^i < 1 < k(i), \quad 1 \le i \le n-1, \quad \text{and} \quad |\lambda_n| < r^n/M^n < 1.
		\end{aligned}
	\end{equation}
	Now it is clear from \eqref{T_i} and \eqref{Lambda_i, Lambda_n Estimate} that
	\begin{equation*}
		\begin{aligned}
			||T_i|| = |\lambda_i| < k(i), \quad 1 \le i \le n-1, \quad \text{and} \quad ||T_n|| = |\lambda_n| < 1.
		\end{aligned}
	\end{equation*}
	
	We are now ready to prove our another main result.
	\begin{thm}\label{Main Theorem 2}
		Let $r$ be as in Lemma \ref{Lem 1}. Then the $n$-tuple of matrices $(T_1, \dots, T_n)$ is a $\Gamma_n$-contraction such that the fundamental operators do not satisfy any condition in \eqref{Commutativity Condition}, whenever $\delta_1, \dots, \delta_n$ are such complex numbers such that $0 < |\delta_i| < r/M$, where $M = \max\{k(i) : 1 \le i \le n\}$, for which $\lambda_i\overline{\lambda}_{n-j} \ne \lambda_j\overline{\lambda}_{n-i}$ for $1 \le i, j \le n-1$ and $\lambda_i \ne 0$ for $1 \le i \le n$.
	\end{thm}
	
	\begin{proof}
		In a similar manner to the proof of Theorem \ref{Main Theorem 1}, one can easily see that the joint spectrum of $(T_1, \dots, T_n) = \{(0, \dots, 0)\} \subseteq \Gamma_n$. Also assume that $f(0, \dots, 0) = 0$ and $||f||_{\infty, \Gamma_n} = 1$.
		
		Since $T^2_i = 0$, by Taylor series expansion of $f$ and the holomorphic functional calculus, we get
		\begin{equation*}
			\begin{aligned}
				f(T_1, \dots, T_n) &=
				\left[
				\begin{matrix}
					0 & \Delta\\
					0 & 0
				\end{matrix}\right],
			\end{aligned}
		\end{equation*}
		where $\Delta = \displaystyle\sum_{i=1}^{n} \lambda_i\frac{\partial f}{\partial z_i}(0, \dots, 0)$. Because $||f||_{\infty, \Gamma_n} = 1$ and $0 < |\delta_i| < r/M$, by applying Proposition \ref{Prop Knese} on $f(rz_1, \dots, rz_n)$ we have
		\begin{equation}\label{Main 3}
			\begin{aligned}
				||f(T_1, \dots, T_n)||
				&= \left|\displaystyle\sum_{i=1}^{n} \lambda_i\frac{\partial f}{\partial z_i}(0, \dots, 0)\right|
				\le r\displaystyle\sum_{i=1}^{n} \left|\frac{\partial f}{\partial z_i}(0, \dots, 0)\right| \le 1 - |f(0, \dots, 0)|^2 = 1.
			\end{aligned}
		\end{equation}
		
		If $||f||_{\infty, \Gamma_n} \neq 1$, then proceeding similarly as in Theorem \ref{Main Theorem 1} to the function $\widetilde{f} = \frac{f}{||f||_{\infty, \Gamma_n}}$, we get
		\begin{equation}\label{Main 4}
			\begin{aligned}
				||\widetilde{f}(T_1, \dots, T_n)||
				&\le 1.
			\end{aligned}
		\end{equation}
		Since $||f||_{\infty, \Gamma_n}$ is constant, $\frac{\partial f}{\partial z_i}(0, \dots, 0)$ is independent of $||f||_{\infty, \Gamma_n}$. Therefore, \eqref{Main 4} yields that
		\begin{equation*}
			\begin{aligned}
				||f(T_1, \dots, T_n)|| &\le ||f||_{\infty, \Gamma_n}.
			\end{aligned}
		\end{equation*}
		Hence, $(T_1, \dots, T_n)$ is a $\Gamma_n$-contraction. Note that
		\begin{equation*}
			\begin{aligned}
				D^2_{T_n} &= I - T^*_nT_n =
				\left[
				\begin{matrix}
					1 & 0\\
					0 & 1
				\end{matrix}\right] -
				\left[
				\begin{matrix}
					0 & 0\\
					\overline{\lambda}_n & 0
				\end{matrix}\right]
				\left[
				\begin{matrix}
					0 & \lambda_n\\
					0 & 0
				\end{matrix}\right] =
				\left[
				\begin{matrix}
					1 & 0\\
					0 & 1 - |\lambda_n|^2
				\end{matrix}\right].
			\end{aligned}
		\end{equation*}
		Thus, we have
		\begin{equation*}
			\begin{aligned}
				D_{T_n} &=
				\left[
				\begin{matrix}
					1 & 0\\
					0 & \sqrt{1 - |\lambda_n|^2}
				\end{matrix}\right].
			\end{aligned}
		\end{equation*}
		Suppose $F_1, \dots, F_{n-1}$ are the fundamental operators of $(T_1, \dots, T_n)$. As $D_{T_n}$ is an invertible matrix, by \eqref{Gamma_n Fundamental} we obtain
		\begin{equation*}
			\begin{aligned}
				F_i &= D^{-1}_{T_n}(T_i - T^*_{n-i}T_n)D^{-1}_{T_n}, \quad 1 \le i \le n-1.
			\end{aligned}
		\end{equation*}
		By routine computation we deduce that
		\begin{equation*}
			\begin{aligned}
				F_i &= 
				\left[
				\begin{matrix}
					0 & \frac{\lambda_i}{\sqrt{1 - |\lambda_n|^2}}\\
					0 & -\frac{\overline{\lambda}_i\lambda_n}{1 - |\lambda_n|^2}
				\end{matrix}\right], \quad 1 \le i \le n-1.
			\end{aligned}
		\end{equation*}
		
		Observe that
		\begin{equation*}
			\begin{aligned}
				F_iF_j &=
				\left[
				\begin{matrix}
					0 & -\frac{\lambda_i\overline{\lambda}_{n-j}\lambda_n}{(1 - |\lambda_n|^2)^{3/2}}\\
					0 & -\frac{\overline{\lambda}_{n-i}\overline{\lambda}_{n-j}\lambda^2_n}{(1 - |\lambda_n|^2)^2}
				\end{matrix}\right] \ne
				\left[
				\begin{matrix}
					0 & -\frac{\lambda_j\overline{\lambda}_{n-i}\lambda_n}{(1 - |\lambda_n|^2)^{3/2}}\\
					0 & -\frac{\overline{\lambda}_{n-j}\overline{\lambda}_{n-i}\lambda^2_n}{(1 - |\lambda_n|^2)^2}
				\end{matrix}\right] = F_jF_i, \quad 1 \le i \ne j \le n-1,
			\end{aligned}
		\end{equation*}
		which shows that $F_i$'s do not commute. Furthermore,
		\begin{equation*}
			\begin{aligned}
				F_iF^*_{n-j} - F_jF^*_{n-i} &=
				\left[
				\begin{matrix}
					\frac{\lambda_j\overline{\lambda}_{n-i} - \lambda_i\overline{\lambda}_{n-j}}{1 - |\lambda_n|^2} & 0\\
					0 & \frac{\overline{\lambda}_j\lambda_{n-i} - \overline{\lambda}_i\lambda_{n-j}}{(1 - |\lambda_n|^2)^2}
				\end{matrix}\right]\\
				&\ne
				\left[
				\begin{matrix}
					0 & 0\\
					0 & \frac{\lambda_j\overline{\lambda}_{n-i} - \lambda_i\overline{\lambda}_{n-j}}{(1 - |\lambda_n|^2)^2} + \frac{(\lambda_i\overline{\lambda}_{n-j} - \lambda_j\overline{\lambda}_{n-i})|\lambda_n|^2}{(1 - |\lambda_n|^2)^2}
				\end{matrix}\right]\\
				&= F^*_{n-j}F_i - F^*_{n-i}F_j
			\end{aligned}
		\end{equation*}
		for $1 \le i, j \le n-1$. Therefore, the $\Gamma_n$-contraction extends to a $\Gamma_n$-isometry, whereas the fundamental operators do not satisfy the conditions in \eqref{Commutativity Condition}. This completes the proof.
	\end{proof}
	
	\begin{rem}\label{Rem 1}
		While Theorem \ref{Main Theorem 1} implies that some conditions in \eqref{Commutativity Condition} are satisfied, the $\Gamma_n$-contraction, presented in Theorem \ref{Main Theorem 2}, does not have fundamental operators that satisfy at least one condition in \eqref{Commutativity Condition}. This means that those conditions are sufficient for constructing an explicit dilation of a given $\Gamma_n$-contraction, whereas the fundamental operators are exempt from satisfying any of them.
	\end{rem}
	
	An explicit construction of $\Gamma_n$-isometric dilation of a $\Gamma_n$-contraction in the presence of commuting fundamental operators is known to us \cite[Theorem 6.6, Corollary 6.7]{A. Pal}, whereas no such development is known for the $\Gamma_n$-contraction with non-commutative fundamental operators. It can be noted that while Theorem \ref{2 times 2 Gamma_n-Contraction} ensures the existence of a dilation of a $2 \times 2$ matrix $\Gamma_n$-contraction, it does not provide an explicit formulation of the dilation. Therefore, would be interesting to construct the dilation for $2 \times 2$ matrix $\Gamma_n$-contractions with non-commutative fundamental operators that we presented in Theorem \ref{Main Theorem 1} and Theorem \ref{Main Theorem 2}.
	
	\vspace{0.3cm}
	
	\noindent{\textbf{Acknowledgement:}} The author thanks Dr. Avijit Pal for carefully reading the first draft of this article.


\begin{thebibliography}{BNS}			
		\bibitem{Agler}
		J. Agler, 
		\textit{Operator Theory and the Caratheodory Metric}, Invent. Math., \textbf{101} (2) (1990), 483--500.
		
		\bibitem{McCarthy 1}
		J. Agler, J. E. McCarthy,
		\textit{Pick Interpolation and Hilbert Function Spaces}, Grad.
		Stud. Math., Amer. Math. Soc., Providence, RI, \textbf{44} (2002).
%
%
		\bibitem{Arveson}
		W. Arveson,
		\textit{Subalgebras of $C^*$-Algebras. II}, Acta Math. \textbf{128} (3--4) (1972) 271--308.
		
		\bibitem{Bhattacharyya}
		T. Bhattacharyya, 
		\textit{The Tetrablock as a Spectral Set,}  Indiana Univ. Math. J., \textbf{63} (2014), 1601--1629.
		
		\bibitem{Bhowmik}
		M. Bhowmik, 
		\textit{A Family of $2 \times 2$ Tetrablock Contraction with Non-commuting Fundamental Operators}, Linear Algebra Appl., \textbf{746} (2026), 205--212.
		
		\bibitem{Biswas}
		S. Biswas and S. Shyam Roy, \textit{Functional models for $\Gamma_n$-contractions and characterization of $\Gamma_n$-isometries}, J. Func. Anal., \textbf{266} (2014), 6224--6255.
		
		\bibitem{Costara}
		C. Costara,
		\textit{On the spectral Nevanlinna-Pick problem}, Studia Math., \textbf{170} (2005), 23--55.
		
		\bibitem{Knese}
		G. Knese,
		\textit{A Schwarz lemma on the Polydisk}, Proc. Am. Math. Soc, \textbf{135} (9) (2007), 2759-2768.
		
%

		\bibitem{Mandal}
		S. Mandal, A.Pal,
		\textit{Necessary Conditions of $\Gamma_n$-Isometry Dilation and the Dilation of a Certain Family of $\Gamma_3$-Contractions}, Complex Anal. Oper. Theory, \textbf{19}, 28 (2025).
				
		\bibitem{Nagy}
		B. Sz.-Nagy, C. Foias, H. Bercovici, L. Kerchy,
		\textit{Harmonic Analysis of Operators on Hilbert Space}, Universitext,  Springer, (2010).
		
		\bibitem{S. Pal}
		S. Pal,
		\textit{The Fundamental Operator Tuples Associated with the Symmetrized Polydisc}, New York J. Math., \textbf{27} (2021), 349-362.
		
		\bibitem{A. Pal}
		A. Pal,
		\textit{On $\Gamma_n$-Contractions and Their Conditional Dilations}, J. Math. Anal. and Appl., {\bf {510}} (2022), 1-36.
		
		\bibitem{paulsen}
		V. Paulsen,
		\textit{Completely Bounded Maps and Operator Algebras,}
		Cambridge  Univ. Press, (2002).
		
		\bibitem{Schaffer}
		J. J. Sch\"{a}ffer,
		\textit{On Unitary Dilations of Contractions}, Proceedings of the American Mathematical Society, (1954).
	\end{thebibliography}
\end{document}